\documentclass[11pt]{article}

\usepackage[a4paper,margin=30mm]{geometry}
\usepackage[T1]{fontenc}
\usepackage[utf8]{inputenc}
\usepackage{lmodern}
\usepackage{microtype}
\usepackage{amsmath,amssymb,amsthm,mathtools}
\usepackage{booktabs,array}
\usepackage{xcolor}
\usepackage{tikz}
\usetikzlibrary{matrix,arrows,positioning,fit,patterns}
\usepackage{enumitem}
\usepackage{hyperref}
\usepackage[capitalise]{cleveref}

\definecolor{linkblue}{RGB}{22,76,130}
\definecolor{shadegray}{RGB}{224,228,232}
\definecolor{diagred}{RGB}{210,92,82}
\hypersetup{colorlinks=true,linkcolor=linkblue,citecolor=linkblue,urlcolor=linkblue,
  pdftitle={Coincidences and Growth of Boxed Mesh Patterns}}
\setlist{nosep}
\newtheorem{theorem}{Theorem}[section]
\newtheorem{proposition}[theorem]{Proposition}
\newtheorem{lemma}[theorem]{Lemma}
\newtheorem{corollary}[theorem]{Corollary}
\theoremstyle{definition}
\newtheorem{definition}[theorem]{Definition}
\newtheorem{conjecture}[theorem]{Conjecture}

\theoremstyle{remark}
\newtheorem{remark}[theorem]{Remark}

\newcommand{\st}{\operatorname{st}}
\newcommand{\Av}{\operatorname{Av}}
\newcommand{\Boxp}{\operatorname{Box}}
\newcommand{\B}{\Boxp}
\newcommand{\boxit}[1]{\boxed{\mathstrut #1}}
\newcommand{\SB}{\mathsf{SB}}
\newcommand{\one}{\mathbf 1}
\newcommand{\Coin}{\mathrel{\equiv_{\rm c}}}

\newcommand{\pattern}[4]{%
  \raisebox{0.6ex}{%
    \begin{tikzpicture}[scale=0.35,baseline=(current bounding box.center),#1]
      \foreach \x/\y in {#4} \fill[gray!50] (\x,\y) rectangle +(1,1);
      \draw (0.01,0.01) grid (#2+0.99,#2+0.99);
      \foreach \x/\y in {#3} \filldraw (\x,\y) circle (6pt);
    \end{tikzpicture}}}

\newcommand{\patternhighlight}[5]{%
  \raisebox{0.6ex}{%
    \begin{tikzpicture}[scale=0.35,baseline=(current bounding box.center),#1]
      \foreach \x/\y in {#4} \fill[gray!50] (\x,\y) rectangle +(1,1);
      \foreach \x/\y in {#5} \fill[diagred!70] (\x,\y) rectangle +(1,1);
      \foreach \x/\y in {#5}
        \fill[pattern=north east lines] (\x,\y) rectangle +(1,1);
      \draw (0.01,0.01) grid (#2+0.99,#2+0.99);
      \foreach \x/\y in {#5} \draw[line width=1.1pt] (\x,\y) rectangle +(1,1);
      \foreach \x/\y in {#3} \filldraw (\x,\y) circle (6pt);
    \end{tikzpicture}}}

\title{Coincidences and Growth of Boxed Mesh Patterns}
\author{Sergey Kitaev$^{1}$, Dun Qiu$^{2}$, and Chao Xu$^{2}$\\[5pt]
\small $^{1}$Department of Mathematics and Statistics, University of Strathclyde,\\[-1pt]
\small Glasgow G1 1XH, United Kingdom\\[-1pt]
\small \texttt{sergey.kitaev@strath.ac.uk}\\[3pt]
\small $^{2}$Center for Combinatorics, LPMC, Nankai University,\\[-1pt]
\small Tianjin 300071, P.~R. China\\[-1pt]
\small \texttt{qiudun@nankai.edu.cn}, \texttt{xclyon1@gmail.com}}

\begin{document}
\maketitle

\begin{abstract}
A boxed  mesh pattern is a mesh pattern whose selected entries lie in an empty
axis-parallel rectangle. We classify coincidences of boxed patterns with classical and vincular
patterns, exhibit a genuinely bivincular coincidence, and prove that no
boxed--bivincular coincidence occurs for patterns of length at least five.
Together with known results, this shows that every boxed pattern of length at
least five has factorial growth and hence fails the Stanley--Wilf property.
At length four, one exceptional orbit is enumerated by the semi-Baxter
numbers, while the remaining exceptional orbit, $\{2143,3412\}$, is
unresolved; we conjecture that it has factorial growth.

For $\Boxp(123)$, we derive an exact maximum-insertion identity and prove the
subfactorial upper bound $2^{5n}n^{\beta n}$, where  $\beta=\log_2(2\cos(\pi/7))<0.85$. A closed enumeration remains open. We also prove a general first-moment
formula for boxed mesh patterns that depends only on the length of the
underlying pattern; in particular, the expected number of $\Boxp(123)$
occurrences in a uniformly random permutation of length $n$ is asymptotic to
$n\log n$.

For $\Boxp(12)$, we identify the occurrence statistic with the up-degree in
the strong Bruhat order, obtaining its maximum, its mean, and an exact
insertion identity for the distribution polynomials. We conjecture that the
coefficients of these polynomials are unimodal.
\end{abstract}



\section{Introduction}
Let $\mathfrak S_n$ denote the set of permutations of
$[n]=\{1,2,\ldots,n\}$, with $\mathfrak S_0=\{\varepsilon\}$. A
\emph{classical pattern} of length $k$ is a permutation
$p=p_1p_2\cdots p_k\in\mathfrak S_k$. An occurrence of $p$ in a
permutation $\pi=\pi_1\pi_2\cdots\pi_n\in\mathfrak S_n$ is a subsequence
$\pi_{i_1}\pi_{i_2}\cdots\pi_{i_k}$, where
$i_1<i_2<\cdots<i_k$, that is order-isomorphic to $p$; that is, $\pi_{i_r}<\pi_{i_s}$ if and only if $p_r<p_s$ for $1\le r,s\le k$.
For example, the pattern $132$ occurs twice in the permutation $24513$,
namely as the subsequences $243$ and $253$.

For a pattern $p$, defined in any specified sense, let
$\Av_n(p)\subseteq\mathfrak S_n$ denote the set of permutations that
avoid $p$, that is, contain no occurrence of $p$.  The \emph{Stanley--Wilf limit} of a pattern \(p\), when it exists, is
\[
 L(p)=\lim_{n\to\infty}|\Av_n(p)|^{1/n}.
\]
Arratia \cite{Arratia1999} proved that an exponential upper bound for
$|\Av_n(p)|$ for a classical pattern $p$ implies the existence of this limit, and Marcus and Tardos
\cite{MarcusTardos2004} proved the required upper bound.  Thus the
Stanley--Wilf theorem says that $L(p)$ exists and is finite for every fixed
classical pattern $p$.  For any pattern $p$, not necessarily classical, we shall say that the
\emph{Stanley--Wilf property} holds if
\[
  \limsup_{n\to\infty}|\Av_n(p)|^{1/n}<\infty.
\]
We say that a positive sequence $(a_n)$ has \emph{factorial growth} if $\log a_n=\Theta(n\log n)$,
or, equivalently, if $a_n=n^{\Theta(n)}$. If $p$ is a consecutive pattern of length at least three, meaning that the
entries in an occurrence must occupy consecutive positions, Elizalde
\cite[Theorem~4.1]{Elizalde2006} proved that there exist constants
$0<c_p<d_p<1$ such that, for all sufficiently large $n$,
\[
  c_p^n n!<|\Av_n(p)|<d_p^n n!.
\]
Consequently its $n$th root tends to infinity.  (Consecutive patterns of
length two are the exceptional trivial cases, with one avoider.)  Vincular patterns, defined in Section~\ref{sec:definitions}, already exhibit
both types of behaviour in length three: for example,
$1\underline{23}$, also written $1\text{-}23$, has the Bell numbers as its
avoidance sequence, whereas some other one-bond patterns have Catalan
avoidance sequences \cite{Claesson2001}.  Since
$B_n^{1/n}\sim n/(e\log n)$, Bell growth also violates the Stanley--Wilf
property.

Mesh patterns, defined in Section~\ref{sec:definitions},  were introduced by Br\"and\'en and Claesson
\cite{BrandenClaesson2011} as a flexible language for imposing geometric
conditions on occurrences of classical permutation patterns.  They have since
developed in several complementary directions: shading and coincidence
criteria \cite{Tenner2013,ClaessonTennerUlfarsson2015}, systematic and
algorithmic classification \cite{HilmarssonEtAl2015,BeanEtAl2023}, exhaustive
generation \cite{HartungEtAl2022}, asymptotic containment
\cite{GovcSmith2022}, and the study of distributions
\cite{KitaevZhang2019,KitaevZhangZhang2020,LvKitaev2025}.

Two patterns are called \emph{coincident} when they impose exactly the same
avoidance condition: every permutation avoids one if and only if it avoids the
other.  Thus coincidence identifies the actual avoidance classes, not merely
their counting sequences; a formal definition is given in
\cref{sec:definitions}.

Avgustinovich, Kitaev and Valyuzhenich \cite{AKV2013} introduced \emph{boxed
patterns}: all $(k-1)^2$ internal cells of a mesh pattern of length $k$ are
shaded.  They showed, among other things, that the permutations avoiding boxed $132$
are exactly those avoiding classical $132$ and are therefore counted by the
Catalan numbers.  They also
showed that an analogue of the Erd\H{o}s--Szekeres theorem fails for boxed
patterns of length greater than two, and they enumerated permutations
avoiding two or more boxed patterns of length three, where generalized
Catalan numbers appear.  More strikingly,
they proved factorial lower bounds for many boxed patterns, showing that the
Stanley--Wilf conclusion for classical patterns does not extend to arbitrary
mesh patterns.  The simplest of their constructions doubles each letter:
\[
 \delta(\pi_1\cdots\pi_m)
 =(2\pi_1)(2\pi_1-1)\cdots(2\pi_m)(2\pi_m-1).
\]
Every pair is a descending pair of consecutive values, and the empty-rectangle
condition shows that $\delta(\pi)$ avoids $\Boxp(123)$: the unselected sibling
of the middle member of any increasing triple lies inside its rectangle.  This
injects $\mathfrak S_m$ into $\Av_{2m}(\Boxp(123))$.  The same observation
and its complement handle any boxed pattern having a monotone interval factor
of length three.  More generally, factorial growth is inherited from any boxed
interval factor; see Proposition \ref{prop:interval-factor}.

Their general proof calls $\tau=\tau_1\cdots\tau_k$ \emph{good} when some
$1<i<k$ has $1<\tau_i<k$ and the factor
$\tau_{i-1}\tau_i\tau_{i+1}$ is nonmonotone.  In the representative $213$
case, every letter of an arbitrary base permutation is replaced by an
increasing pair of consecutive values.  Whichever sibling plays $\tau_i$, its
mate either lies in the forbidden rectangle or forces the adjacent member of
the pattern to be the mate, contradicting the required order.  Reverse and
complement deal with the other three nonmonotone triples.  If $\tau$ is not
good, its entries split into monotone blocks around its minimum and maximum.
Either it has a monotone factor of length three, in which case the
$\Boxp(123)$ or $\Boxp(321)$ construction applies, or it is one of the four
exceptional permutations below.  Together these ideas give factorial lower
bounds for twenty of the twenty-four underlying permutations of length four.
Their argument left
\begin{equation}\label{eq:four}
 2143,\qquad 3412,\qquad 2413,\qquad 3142,
\end{equation}
and treated them as a single symmetry class.

It is important that the same theorem already settles every larger length.
More precisely, \cite[Theorem~3]{AKV2013} proves
\begin{equation*}\label{eq:akv-factorial}
 |\Av_n(\Boxp(\tau))|\ge \left\lfloor\frac n2\right\rfloor!
\end{equation*}
for every $\tau\in\mathfrak S_k$, $k\ge4$, other than the four permutations in
\eqref{eq:four}.  Consequently every boxed pattern of length at least five
fails the Stanley--Wilf property.  Thus the obstruction to coincidence with a
classical, vincular or bivincular pattern in lengths at least five is
accompanied by a much stronger growth statement, but the two facts have
logically independent proofs: noncoincidence alone would not imply factorial
growth.

Three questions from the final section of \cite{AKV2013} are directly relevant
here: determine the Stanley--Wilf behaviour of the four patterns in
\eqref{eq:four}; enumerate $\Boxp(123)$; and decide when a boxed pattern can be
coincident with a classical pattern.  The last question has since been settled:
coincident mesh patterns have the same underlying permutation
\cite[Lemma~4.1]{ClaessonTennerUlfarsson2015}, and the enclosed-diagonal
criterion of Tenner \cite{Tenner2013} classifies when a mesh is superfluous.
For boxed patterns the complete list is recalled in
Proposition \ref{prop:classical-coincidence}.  The first two questions remain the main
enumerative motivation for this paper.

Boxed patterns are also called \emph{frame patterns}; under that name their
occurrence statistics have been studied in the cycle structure of
permutations by Jones, Kitaev and Remmel \cite{JonesKitaevRemmel2015}.  On
the algorithmic side, deciding boxed containment has been studied by Bruner
and Lackner \cite{BrunerLackner2013}, by Cho, Na and Sim \cite{ChoNaSim2015}
and by Amit, Bille, Hagge Cording, Li G\o rtz and Wedel Vildh\o j
\cite{AmitEtAl2016}; see Remark \ref{rem:matching}.

The four exceptional cases separate in two ways.  First, the four patterns in
\eqref{eq:four} are treated together in \cite[p.~48]{AKV2013}, but they do
not form a single symmetry class: reverse, complement and inverse split
\eqref{eq:four} into
\begin{equation}\label{eq:orbits}
 \{2143,3412\}\qquad\text{and}\qquad\{2413,3142\}.
\end{equation}
The two avoidance sequences are different: at size five they have respectively
$106$ and $104$ elements.  Second, the latter orbit is not a counterexample to
the Stanley--Wilf property.  The key observation for this orbit, that we prove in Lemma~\ref{lem:tightening}, is the
coincidence
\[
  \Boxp(2413)\Coin 2\underline{41}3,
\]
where the underline requires the entries playing $4$ and $1$ to be adjacent;
the equivalent dash notation is $2$-$41$-$3$.
The right-hand class consists of the semi-Baxter permutations, completely
enumerated by Bouvel, Guerrini, Rechnitzer and Rinaldi
\cite{BouvelEtAl2018}.  This settles one of the two orbits in
\eqref{eq:orbits} and explains the value $104$ at size five.  We also classify
all coincidences of a boxed pattern with a nonclassical vincular pattern: none
exist in lengths five or more.  Moreover, $\Boxp(2413)$ is coincident with a
genuinely bivincular pattern, and we prove that no boxed pattern of length at
least five is coincident with any bivincular pattern.  

For the unresolved orbit $\{2143,3412\}$, we show that replacing every entry
$x$ by the increasing pair $(2x-1)(2x)$ preserves the number of boxed $2143$
occurrences. We also conjecture that the avoidance sequence associated with
this orbit has factorial growth; see Conjectures~\ref{conj:interleave}
and~\ref{conj:2143-factorial} and the accompanying discussion in
Section~\ref{sec:remaining}.

Our second focus is boxed $123$.  The sequence is recorded as OEIS A201168
\cite{OEIS201168}, but no generating function or coefficient formula is given
there or, to our knowledge, elsewhere in the literature.  We give an exact
maximum-insertion identity and show that its prefix condition is itself
enumerated by $(t-1)!$; this gives the five new terms $a_{13},\ldots,a_{17}$.
A weighted halving argument also gives
\[
 |\Av_n(\Boxp(123))|\le 2^{5n}n^{\beta n},
 \qquad \beta=\log_2\!\left(2\cos\frac{\pi}{7}\right)<0.85,
\]
so the avoidance class is subfactorial even though it does not have the
Stanley--Wilf property.  The identity does not close in the total numbers,
and a closed enumeration remains open.  On the distributional side, we prove
that the expected number of occurrences of $\Boxp(\tau)$ in a uniformly random
permutation depends only on the length of $\tau$, not on $\tau$ itself.  For
boxed $123$ the exact mean is $(n+2)H_n-3n\sim n\log n$, so a random
permutation contains on average on the order of $n\log n$ such occurrences.

A complementary distribution question concerns the boxed pattern $12$.  A
recent sequence of papers has completed the classification of length-two mesh
patterns by distribution-equivalence and Wilf-equivalence.  Su, Kitaev and
Zhang \cite{SuKitaevZhang2026} obtained a near-classification; Fang, Fu,
Kitaev, Li, Su and Sun \cite{FangEtAl2026} reduced both classifications to a
single remaining case; and Zhang and Zhao \cite{ZhangZhao2026} settled that
case.  The first of these papers records the initial distribution polynomials
for boxed $12$ but leaves a general formula open.  We observe that the number of
boxed $12$ occurrences is exactly the up-degree of a permutation in the strong
Bruhat order.  The standard Bruhat cover criterion, also discussed in
mesh-pattern language by Bouvel, Ferrari and Tenner
\cite{BouvelFerrariTenner2024}, therefore brings the results of Adin and
Roichman \cite{AdinRoichman2006} into the problem.  In particular, it gives
the maximum number of occurrences, the number of maximizers, and the exact
mean.  We add a short direct proof of the mean and an insertion identity that
isolates the catalytic prefix statistic needed for the full distribution.

The paper is organized as follows. In Section~\ref{sec:definitions}, we
introduce the notation and recall the preliminary results used throughout the
paper. In Section~\ref{sec:interval}, we translate boxed containment into
consecutive containment after restriction to an interval of values. In
Section~\ref{sec:length4}, we correct the length-four symmetry reduction,
identify the semi-Baxter orbit, and record the resulting global growth
classification. In Section~\ref{sec:coincidence}, we classify classical and
vincular coincidences, exhibit a genuinely bivincular coincidence, and prove
that no boxed--bivincular coincidence occurs in length at least five. In
Section~\ref{sec:123}, we give an exact identity, a subfactorial upper bound,
and a general first-moment formula motivated by boxed $123$. In
Section~\ref{sec:remaining}, we discuss the remaining orbit and state two
conjectures. Finally, in Section~\ref{sec:box12}, we connect boxed $12$ with
covers in the strong Bruhat order and record the resulting distributional
information.

\section{Preliminaries}\label{sec:definitions}
If $w=w_1\cdots w_k$ is a word with distinct letters, its
\emph{standardization}, also called its \emph{reduced form}, is the unique
permutation $\st(w)\in\mathfrak S_k$ whose entries have the same relative
order as those of $w$. Recall that a permutation $\pi$ \emph{contains} the
classical pattern $\tau\in\mathfrak S_k$ if $\st(\pi_{i_1}\cdots\pi_{i_k})=\tau$ for some $i_1<\cdots<i_k$; otherwise, $\pi$ avoids $\tau$.

A \emph{mesh pattern} of length $k$ is a pair $P=(\tau,R)$, where
$\tau\in\mathfrak S_k$ and $R\subseteq\{0,1,\ldots,k\}^2$. Its diagram consists of the plot of $\tau$, with points
$(r,\tau_r)$ for $1\le r\le k$, together with a shaded unit square $[a,a+1]\times[b,b+1]$ for each $(a,b)\in R$. For example,
\[
  P=\bigl(2413,\{(0,1),(2,2),(3,0)\}\bigr)
  =
  \pattern{scale=0.52}{4}
    {1/2,2/4,3/1,4/3}
    {0/1,2/2,3/0}.
\]

Let $\pi=\pi_1\cdots\pi_n\in\mathfrak S_n$, and suppose that
$i_1<\cdots<i_k$ determine a classical occurrence of $\tau$ in $\pi$.
Let $v_1<\cdots<v_k$ be the selected values $\pi_{i_1},\ldots,\pi_{i_k}$ written in increasing
order, and set $i_0=v_0=0$ and $i_{k+1}=v_{k+1}=n+1$. The selected entries form an \emph{occurrence} of the mesh pattern
$P=(\tau,R)$ if, for every $(a,b)\in R$, there is no plot point
$(j,\pi_j)$ satisfying $i_a<j<i_{a+1}$ and $v_b<\pi_j<v_{b+1}$. Thus each shaded square prescribes a rectangular region that must contain
no other point of the plot of $\pi$. When $R=\varnothing$, this definition
reduces to classical pattern containment.

The \emph{boxed pattern} with underlying permutation $\tau\in\mathfrak S_k$
is $\Boxp(\tau)=\bigl(\tau,\{1,\ldots,k-1\}^2\bigr)$. It is also denoted by a box drawn around the permutation, such as
$\boxit{123}$, or by displaying its mesh:
\[
 \Boxp(123)=\boxit{123}=
 \pattern{scale=0.52}{3}{1/1,2/2,3/3}{1/1,1/2,2/1,2/2}.
\]
We use the notation $\Boxp(123)$ throughout, as it is more convenient in longer formulas.

For a permutation \(\pi=\pi_1\cdots\pi_n\in \mathfrak S_n\), its \emph{reverse},
\emph{complement}, and \emph{inverse} are defined, respectively, by $\pi^{\mathrm r}_i=\pi_{n+1-i}$, $\pi^{\mathrm c}_i=n+1-\pi_i$,
 $\pi^{-1}_{\pi_i}=i$ for $1\leq i\leq n$.
Geometrically, these operations reflect the plot in a vertical line, reflect
it in a horizontal line, and reflect it in the main diagonal, respectively.
They act on mesh patterns by applying the same reflection to both the
underlying permutation and the shaded cells.  In particular, if
\(P=(\tau,R)\) has length \(k\), then
\[
\begin{aligned}
 P^{\mathrm r}
   &=\bigl(\tau^{\mathrm r},
       \{(k-a,b):(a,b)\in R\}\bigr),\\
 P^{\mathrm c}
   &=\bigl(\tau^{\mathrm c},
       \{(a,k-b):(a,b)\in R\}\bigr),\\
 P^{-1}
   &=\bigl(\tau^{-1},
       \{(b,a):(a,b)\in R\}\bigr).
\end{aligned}
\]
These symmetries preserve containment, avoidance, and coincidence.  They also
preserve boxed patterns: $\Boxp(\tau)^{\mathrm r}=\Boxp(\tau^{\mathrm r})$, $\Boxp(\tau)^{\mathrm c}=\Boxp(\tau^{\mathrm c})$, and $\Boxp(\tau)^{-1}=\Boxp(\tau^{-1})$.

If $X\subseteq[k-1]$, a \emph{vincular pattern} $(\tau,X)$ requires
$i_{x+1}=i_x+1$ for every $x\in X$.  We indicate bonds by underlining; thus
$2\underline{41}3$ is the pattern for which the entries playing $4$ and $1$
are adjacent.  The notation $2$-$41$-$3$ is equivalent.  A
\emph{bivincular pattern} $(\tau,X,Y)$ additionally requires that the selected
values of ranks $y$ and $y+1$ be consecutive for each $y\in Y$
\cite{BousquetMelouEtAl2010}.  Vincular and
bivincular patterns are mesh patterns: position bonds shade complete columns,
and value bonds shade complete rows.  For background on all of these families
see Kitaev \cite{Kitaev2011} and the survey of Steingr\'imsson
\cite{Steingrimsson2010}.

For any pattern $P$,  let $\Av(P)=\bigcup_{n\ge0}\Av_n(P)$.  Two patterns $P$ and $Q$ are
\emph{coincident}, written $P\Coin Q$, if $\Av_n(P)=\Av_n(Q)$ for every $n\ge0$. Equivalently, every permutation contains $P$ if and only if it contains $Q$.
In informal terms, the two pattern-avoidance tests are indistinguishable even
though the patterns may be drawn differently.  This is stronger than
\emph{Wilf equivalence}, which asks only for
$|\Av_n(P)|=|\Av_n(Q)|$ in every length.

We shall also use the obstruction to coincidence with classical patterns
introduced by Tenner \cite{Tenner2013}.  Write
$G(\tau)=\{(i,\tau_i):1\le i\le k\}$ for the graph of
$\tau\in\mathfrak S_k$.  A shaded square $(a,b)\in R$ is a
\emph{pointless square}, or an enclosed diagonal of length one, if its closed
square has no vertex in $G(\tau)$.  For $c\ge1$, a set of shaded squares
\[
 D=\{(a+i,b+i):0\le i\le c\}\subseteq R
\]
is an \emph{enclosed NE-diagonal} if
\[
 G(\tau)\cap\{(a+i,b+i):0\le i\le c+1\}
 =\{(a+i,b+i):1\le i\le c\}.
\]
Similarly,
$D=\{(a+i,b-i):0\le i\le c\}\subseteq R$ is an
\emph{enclosed SE-diagonal} if
\[
 G(\tau)\cap\{(a+i,b+1-i):0\le i\le c+1\}
 =\{(a+i,b+1-i):1\le i\le c\}.
\]
Either type has length $c+1$ and is called \emph{proper} when its length is
greater than one.  This is the square-based formulation in
\cite[Definition~3.1]{ClaessonTennerUlfarsson2015}. For a mesh pattern $P=(\tau,R)$, let $\operatorname{enc}(P)$ denote its set
of enclosed diagonals, regarded as sets of shaded squares.

\begin{proposition}[\cite{ClaessonTennerUlfarsson2015}, Lemmas~4.1 and 4.2]
\label{prop:coincidence-invariants}
If the mesh patterns $P=(\tau,R)$ and $Q=(\sigma,S)$ are coincident, then $\tau=\sigma$ and $\operatorname{enc}(P)=\operatorname{enc}(Q)$.
\end{proposition}

Equality of enclosed diagonals is necessary but is not sufficient for general
mesh-pattern coincidence \cite{ClaessonTennerUlfarsson2015}.  We shall use the
necessary direction explicitly in the proof of
\cref{thm:vincular-classification}.

A mesh $R$ is called \emph{superfluous} for $\tau$ if its shading imposes no
additional avoidance condition, that is,
$(\tau,R)\Coin(\tau,\varnothing)$.  Thus ``superfluous'' describes the whole
mesh; it does not mean that its cells can necessarily be removed independently.

\begin{theorem}[\cite{Tenner2013}]
\label{thm:superfluous-mesh}
The mesh $R$ is superfluous for $\tau$ if and only if
$\operatorname{enc}((\tau,R))=\varnothing$.
\end{theorem}

\begin{figure}[ht]
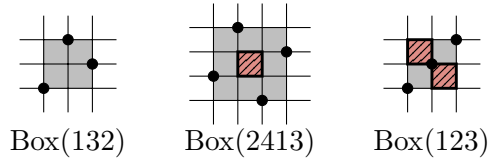

\centering
\begin{tabular}{c@{\qquad}c@{\qquad}c}
$\pattern{scale=0.92}{3}{1/1,2/3,3/2}{1/1,1/2,2/1,2/2}$
&$\patternhighlight{scale=0.92}{4}{1/2,2/4,3/1,4/3}{1/1,1/2,1/3,2/1,2/2,2/3,3/1,3/2,3/3}{2/2}$
&$\patternhighlight{scale=0.92}{3}{1/1,2/2,3/3}{1/1,1/2,2/1,2/2}{1/2,2/1}$\\
$\Boxp(132)$&$\Boxp(2413)$&$\Boxp(123)$
\end{tabular}
\caption{Enclosed diagonals in boxed meshes.  Enclosed diagonals are drawn
hatched and outlined as well as coloured, so that they remain distinguishable
in grayscale.  The left-hand mesh has none; the marked central square in the
middle is a pointless enclosed diagonal; the two marked squares on the right
form a proper SE-diagonal of length two.}
\label{fig:enclosed-diagonals}
\end{figure}

Thus boxed patterns are \emph{not} generally free of enclosed diagonals.
Indeed, Tenner's classification \cite[Corollary~3.8]{Tenner2013} says that the
only boxed patterns without one are those with underlying permutations
$1,12,21,132,213,231,$ and $312$.  In particular, every boxed pattern of
length at least four has an enclosed diagonal.

\section{Boxed patterns as interval-consecutive patterns}\label{sec:interval}

For an integer interval $I=[a,b]$, write $\pi|_I$ for the subword of $\pi$
obtained by retaining entries whose values lie in $I$.

\begin{definition}
Let $\tau\in\mathfrak S_k$.  An occurrence
$\pi_{i_1}\cdots\pi_{i_k}$ of $\tau$ in $\pi$ is a \emph{boxed occurrence} if
the open axis-parallel rectangle determined by its leftmost, rightmost, lowest
and highest selected points contains no further point of the plot of $\pi$.
The corresponding mesh pattern is denoted $\B(\tau)$.
\end{definition}

Because a permutation plot has no repeated abscissae or ordinates, the
condition can equivalently be stated with the closed rectangle: the closed
bounding rectangle contains no plot point other than the selected ones.

\begin{theorem}\label{thm:interval}
For every $\tau\in\mathfrak S_k$ and every permutation $\pi$, the following are
equivalent.
\begin{enumerate}[label=\textup{(\roman*)}]
\item $\pi$ contains $\B(\tau)$;
\item for some interval $I$ of values, the word $\pi|_I$ has a consecutive
factor whose standardization is $\tau$.
\end{enumerate}
\end{theorem}

\begin{proof}
Suppose first that points with positions $i_1<\cdots<i_k$ form a boxed
occurrence, and let $u$ and $v$ be the minimum and maximum of their values.
No other point with value in $[u,v]$ can have position strictly between
$i_1$ and $i_k$.  Thus the selected entries form a consecutive factor of
$\pi|_{[u,v]}$.

Conversely, suppose a consecutive factor of $\pi|_I$ standardizes to $\tau$.
Let $u$ and $v$ be the minimum and maximum values in this factor.  Any extra
point in its bounding rectangle would have value in $[u,v]\subseteq I$ and
position between the first and last entries of the factor.  It would therefore
appear inside the factor of $\pi|_I$, a contradiction.
\end{proof}

The criterion places boxed avoidance beside simsun-type restrictions, but with
all value intervals rather than only initial intervals.  It also gives a uniform
way to search for occurrences: restrict to intervals, then use consecutive
pattern matching.  A second equivalent formulation, used in
\cite{AKV2013}, is obtained by repeatedly deleting the leftmost, rightmost,
smallest or largest entry: $\pi$ contains $\B(\tau)$ if and only if some such
sequence of external deletions ends at $\tau$.

\begin{remark}
\label{rem:matching}
Deciding whether a given permutation contains a given boxed pattern has also
been studied algorithmically.  Bruner and Lackner \cite{BrunerLackner2013}
devote a section to boxed mesh patterns and give an $O(n^3)$ algorithm; Cho,
Na and Sim \cite{ChoNaSim2015} improve this to $O(n^2m)$ and then to
$O(n^2\log m)$, where $m$ is the length of the pattern; and Amit, Bille, Hagge Cording, Li G\o rtz and Wedel
Vildh\o j \cite{AmitEtAl2016} give an $O(n^2)$ algorithm, which is optimal
because the output can have size $\Omega(n^2)$.  
\end{remark}

The interval formulation also gives a useful inheritance principle.  Say that
$\sigma\in\mathfrak S_r$ is an \emph{interval factor} of
$\tau=\tau_1\cdots\tau_k$ if, for some $j$ and $a$,
\[
 \{\tau_j,\tau_{j+1},\ldots,\tau_{j+r-1}\}
   =\{a,a+1,\ldots,a+r-1\}
\]
and $\sigma=\st(\tau_j\cdots\tau_{j+r-1})$.  Thus the entries forming
$\sigma$ occupy consecutive positions and consecutive ranks in $\tau$.

\begin{proposition}\label{prop:interval-factor}
If $\sigma$ is an interval factor of $\tau$, then every occurrence of
$\B(\tau)$ contains an occurrence of $\B(\sigma)$.  Consequently, $\Av_n(\B(\sigma))\subseteq \Av_n(\B(\tau))$ for $n\ge0$. In particular, failure of the Stanley--Wilf property for $\B(\sigma)$ implies
failure of the property for $\B(\tau)$.
\end{proposition}

\begin{proof}
Take the selected points of a boxed $\tau$ occurrence that correspond to the
factor $\tau_j\cdots\tau_{j+r-1}$.  No other selected point lies horizontally
between the first and last of these points, because the positions of the
factor are consecutive.  No other selected point lies vertically between
their lowest and highest points, because their ranks are consecutive.  An
unselected point in the smaller bounding rectangle would also lie in the
bounding rectangle of the original boxed occurrence.  Hence the selected
factor is a boxed $\sigma$ occurrence.  Taking contrapositives gives the class
inclusion, and the final assertion follows by comparing cardinalities.
\end{proof}

For example, Proposition \ref{prop:interval-factor} and the doubled-letter construction
show immediately that $\B(\tau)$ has factorially many avoiders whenever
$\tau$ has an interval factor $123$ or $321$.  The ``good pattern'' argument
of \cite{AKV2013} handles many patterns without such a monotone interval
factor and is therefore genuinely stronger than this inheritance principle.

\section{The exceptional length-four patterns}\label{sec:length4}

Reverse, complement and inverse preserve boxed containment and avoidance.  A
direct calculation gives the two orbits in \eqref{eq:orbits}; in particular,
inverse fixes $2143$ and interchanges $2413$ with $3142$.  The distinction is
enumerative as well as formal:
\begin{center}
\begin{tabular}{c@{\qquad}rrrrrr}
\toprule
$n$ & 0&1&2&3&4&5\\
\midrule
$|\Av_n(\B(2143))|$ &1&1&2&6&23&106\\
$|\Av_n(\B(2413))|$ &1&1&2&6&23&104\\
\bottomrule
\end{tabular}
\end{center}

Recall that an occurrence of the vincular pattern $2\underline{41}3$ is a
subsequence $a,b,c,d$ in that positional order, with $b$ and $c$ adjacent and $c<a<d<b$.

\begin{lemma}\label{lem:tightening}
A permutation contains the boxed pattern $\B(2413)$ if and only if it contains
the vincular pattern $2\underline{41}3$.
\end{lemma}

\begin{proof}
Suppose $a,b,c,d$ is a boxed $2413$ occurrence, so
$c<a<d<b$.  Read the entries between $b$ and $c$.  Since the
rectangle is empty, each is either above $b$ or below $c$.  Beginning with $b$
and ending with $c$, there is an adjacent transition $b',c'$ from the upper to
the lower region.  Since $b'\ge b$ and $c'\le c$, $c<a<d<b$ gives
$c'<a<d<b'$; and $b'$ and $c'$ lie weakly between $b$ and $c$ in position, so
$a$ precedes $b'$ and $d$ follows $c'$.  Then $a,b',c',d$ is an occurrence of
$2\underline{41}3$.

For the converse, use the external-deletion characterization in \cite{AKV2013} discussed in Section~\ref{sec:interval}. It is enough to show that whenever a
permutation of length greater than four contains $2\underline{41}3$, one of
its four external entries can be deleted while preserving such an occurrence.
Choose an occurrence $a,b,c,d$.  If an external entry is not among these four,
delete it; the bond survives, since $b$ and $c$ are adjacent and therefore no
entry lies between them.  The only remaining case is therefore that $a,b,c,d$
are respectively the leftmost, largest, smallest and rightmost entries of the
whole permutation; the assignment of roles is forced, because by
$c<a<d<b$ the largest of the four is $b$ and the smallest is $c$,
while in positional order the first is $a$ and the last is $d$.
Let $x$ be the entry immediately preceding $b$.

If \(x>d\), deleting the largest entry \(b\) makes \(x\) adjacent to \(c\), and \(a,x,c,d\) is still an occurrence. Suppose instead that \(x<d\). If \(x\neq a\), then \(c<x<d<b\), so \(x,b,c,d\) is an occurrence and the leftmost entry \(a\) may be deleted. The only remaining possibility is \(x=a\). Since \(a\) is the first entry and \(b\) and \(c\) are adjacent, the permutation begins with \(a,b,c\). Moreover, \(d\) is the last entry, so, because the permutation has length greater than four, the entry \(y\) immediately following \(c\) exists and satisfies \(y\neq d\). Since \(c\) and \(b\) are respectively the smallest and largest entries of the permutation, \(c<y<b\). If \(y>a\), then \(c<a<y<b\), so \(a,b,c,y\) is an occurrence and \(d\) may be deleted. If \(y<a\), then \(y<a<d<b\); deleting the smallest entry \(c\) makes \(b\) and \(y\) adjacent, and \(a,b,y,d\) is an occurrence. Thus an external entry can always be deleted while preserving an occurrence. Repeating the argument ends at the permutation \(2413\), proving boxed containment.
\end{proof}

\begin{remark}
Contrary to what one might first expect, boxed meshes are not generally free of
enclosed diagonals; see Figure~\ref{fig:enclosed-diagonals}.  The two meshes in
Lemma \ref{lem:tightening} have the same nonempty set of enclosed diagonals: the
pointless square $(2,2)$.  Thus the enclosed-diagonal invariant of
\cite{ClaessonTennerUlfarsson2015} is consistent with the coincidence, but
Tenner's superfluous-mesh theorem \cite{Tenner2013} does not prove it: that
theorem applies precisely when there are no enclosed diagonals.  Nor are the
two meshes nested, so the one-square Shading Lemma
\cite[Lemma~3.11]{HilmarssonEtAl2015} cannot be applied directly.  Its
simultaneous extension \cite[Section~7]{ClaessonTennerUlfarsson2015} is also a
sufficient local shading criterion, not a converse that would establish this
coincidence.  Finally, the coincidence theorem for two vincular patterns does
not apply because the boxed mesh is not vincular.  The tightening argument
above is therefore not an immediate consequence of these earlier results; it
supplies the required global movement of the adjacency bond.
\end{remark}

\begin{theorem}\label{thm:semibaxter}
For every $n\geq0$, $|\Av_n(\B(2413))|=|\Av_n(\B(3142))|=\SB_n$, where $\SB_n$ is the $n$th semi-Baxter number (OEIS A117106).  With
$\SB_0=\SB_1=1$, these numbers satisfy
\begin{equation}\label{eq:sbrec}
 \SB_n=\frac{11n^2+11n-6}{(n+4)(n+3)}\SB_{n-1}
 +\frac{(n-3)(n-2)}{(n+4)(n+3)}\SB_{n-2}\qquad(n\ge2),
\end{equation}
and, for $n\ge2$,
\begin{equation}\label{eq:sbformula}
 \SB_n=\frac{24}{(n-1)n^2(n+1)(n+2)}
 \sum_{j=0}^{n}\binom{n}{j+2}\binom{n+2}{j}
 \binom{n+j+2}{j+1}.
\end{equation}
\end{theorem}

\begin{proof}
By Lemma~\ref{lem:tightening}, the avoiders of \(\Boxp(2413)\) are the
semi-Baxter permutations, and hence have cardinality \(\SB_n\).
Inverse symmetry gives the same enumeration for \(\Boxp(3142)\). The recurrence and sum are the semi-Baxter enumeration of
Bouvel et al.\ \cite[Proposition~13 and Theorem~14]{BouvelEtAl2018}.
\end{proof}

In particular, $\SB_n\sim A\mu^n n^{-6}$ with
$\mu=((\sqrt5-1)/2)^{-5}$ and an explicit positive constant $A$
\cite{BouvelEtAl2018}.  Boxed $2413$ therefore satisfies the Stanley--Wilf
property.  It also follows that no injection from $\mathfrak S_{\lfloor n/2\rfloor}$
to $\Av_n(\B(2413))$ can exist for every $n$: such an injection would force a
factorial lower bound on an exponentially bounded sequence.

Combining \cref{thm:semibaxter} with the factorial constructions in
\cite{AKV2013} leaves exactly one unresolved dihedral orbit at length four,
namely $\{2143,3412\}$.

The combination is worth recording as a single global statement.  It also
clarifies that coincidence and growth are separate questions: the
noncoincidence theorem in \cref{sec:coincidence} is not needed for the
factorial lower bound in large lengths.

\begin{theorem}
\label{thm:global-growth}
Let $\tau\in\mathfrak S_k$ with $k\ge4$.
\begin{enumerate}[label=\textup{(\roman*)}]
\item If
$\tau\notin\{2143,3412,2413,3142\}$, then $|\Av_n(\B(\tau))|\ge \left\lfloor\frac n2\right\rfloor!$,
so $\B(\tau)$ does not have the Stanley--Wilf property.
\item If $\tau\in\{2413,3142\}$, then $\B(\tau)$ has the Stanley--Wilf
property and is enumerated by the semi-Baxter numbers.
\item The Stanley--Wilf property remains open for
$\tau\in\{2143,3412\}$; these two cases are equivalent by complement.
\end{enumerate}
In particular, no boxed pattern of length at least five has the
Stanley--Wilf property.
\end{theorem}

\begin{proof}
Part~\textup{(i)} is \cite[Theorem~3]{AKV2013}.  Stirling's formula gives $\left(\left\lfloor\frac n2\right\rfloor!\right)^{1/n}
 \sim \sqrt{\frac{n}{2e}}$,
which tends to infinity.  Part~\textup{(ii)} is
\cref{thm:semibaxter}, and part~\textup{(iii)} follows from complement
symmetry and the fact that neither of the preceding results applies to this
orbit.  For $k\ge5$ none of the four exceptional permutations in
part~\textup{(i)} has the required length, proving the last assertion.
\end{proof}

Together with the short cases, \cref{thm:global-growth} gives an almost
complete classification over all boxed patterns.  The Stanley--Wilf property
is known to hold for the boxed patterns with underlying permutations $1,12,21,132,213,231,312,2413,3142$;
it is known to fail for every other boxed pattern except $2143$ and $3412$.

\section{Classical, vincular and bivincular coincidences}
\label{sec:coincidence}

We first record the classical case, which is now a consequence of the general
coincidence theory for mesh patterns.

\begin{proposition}
\label{prop:classical-coincidence}
Let $\tau\in\mathfrak S_k$.  Then $\Boxp(\tau)$ is coincident with a classical
pattern if and only if $k\le2$, or $k=3$  and  $\tau\in\{132,213,231,312\}$. In every case the classical pattern is $\tau$ itself.  In particular, no boxed
pattern of length at least four is coincident with a classical pattern.
\end{proposition}

\begin{proof}
By Proposition \ref{prop:coincidence-invariants}, a coincident classical pattern would
have to be $\tau$ itself.  By \cref{thm:superfluous-mesh}, coincidence with
$\tau$ holds exactly when the full internal box has no enclosed diagonal.
Applying that criterion gives precisely the displayed list; this is also
\cite[Proposition~1]{AKV2013} and
\cite[Corollary~3.8]{Tenner2013}.
\end{proof}

The answer changes for vincular patterns, but only in small lengths.

\begin{theorem}
\label{thm:vincular-classification}
Let $V$ be a vincular pattern of length $k$ with at least one bond.  Then
$\Boxp(\tau)\Coin V$ if and only if the pair occurs in the following list:
\begin{align*}
k=2:\quad&
 \Boxp(12)\Coin\underline{12},
 &\Boxp(21)\Coin\underline{21};\\
k=3:\quad&
 \Boxp(132)\Coin\underline{13}2,
 &\Boxp(213)\Coin2\underline{13},\\
&\Boxp(231)\Coin2\underline{31},
 &\Boxp(312)\Coin\underline{31}2;\\
k=4:\quad&
 \Boxp(2413)\Coin2\underline{41}3,
 &\Boxp(3142)\Coin3\underline{14}2.
\end{align*}
In particular, there is no boxed--vincular coincidence in length $k\ge5$.
\end{theorem}

\begin{proof}
Write $V=(\sigma,X)$.  If $\Boxp(\tau)\Coin V$, the first necessary condition
in \cref{prop:coincidence-invariants} gives $\sigma=\tau$, while the second
gives
\begin{equation}\label{eq:same-enclosed-diagonals}
 \operatorname{enc}(\Boxp(\tau))=\operatorname{enc}(V).
\end{equation}
We use this equality repeatedly.

Fix a bond $c\in X$.  It shades the complete column $c$ in $V$, including the
two boundary cells $(c,0)$ and $(c,k)$, neither of which is shaded in the boxed
mesh.  If $(c,0)$ touched no point of $G(\tau)$, it would be a pointless
enclosed diagonal of $V$ that is absent from $\Boxp(\tau)$, contradicting
\eqref{eq:same-enclosed-diagonals}.  It must therefore touch one of
$(c,1)$ and $(c+1,1)$, so one of $\tau_c,\tau_{c+1}$ is $1$.  Applying the
same argument to $(c,k)$ shows that one of them is $k$.  Hence
\begin{equation}\label{eq:extreme-bond}
 \{\tau_c,\tau_{c+1}\}=\{1,k\}.
\end{equation}
Since the values $1$ and $k$ occur only once, two distinct bonds cannot both
satisfy \eqref{eq:extreme-bond}; consequently $X=\{c\}$.

Applying complement to both patterns if necessary, we may assume that $\tau_c=1$ and $\tau_{c+1}=k$. Recall that \(V\) shades only column \(c\).  Every internal cell outside
column \(c\) must therefore touch a point of \(G(\tau)\); otherwise it would
be a pointless enclosed diagonal of \(\Boxp(\tau)\) that is absent from \(V\),
contradicting \eqref{eq:same-enclosed-diagonals}.

Suppose first that \(c\le k-2\), so that column \(c+1\) is an internal column.
Its \(k-1\) internal cells can be touched only by the two plot points $(c+1,\tau_{c+1})=(c+1,k)$ and $(c+2,\tau_{c+2})$. The first point touches only one internal cell of column \(c+1\), while the
second touches at most two.  Thus at most three of the \(k-1\) internal cells
in this column can touch \(G(\tau)\).  Since \(k\ge5\), at least one of these
cells is pointless, a contradiction.  Hence \(c=k-1\).

Similarly, if \(c\ge2\), consider the internal column \(c-1\).  Its cells can
be touched only by $(c-1,\tau_{c-1})$ and $(c,\tau_c)=(c,1)$. Again, the extreme-valued point touches only one internal cell and the other
point touches at most two.  For \(k\ge5\), this leaves a pointless cell in
column \(c-1\), so \(c=1\).

We have obtained both \(c=k-1\) and \(c=1\), which would force \(k=2\).
This contradicts \(k\ge5\), and therefore no boxed--vincular coincidence can
occur in length at least five.

It remains to settle $k\le4$.

For $k=2$, condition \eqref{eq:extreme-bond} leaves only
$\underline{12}$ and $\underline{21}$.  A permutation contains a classical
$12$ if and only if it has an adjacent ascent: a permutation with no adjacent
ascent is decreasing.  Hence $12\Coin\underline{12}$, and similarly
$21\Coin\underline{21}$.  Combining these observations with
\cref{prop:classical-coincidence} gives the two asserted boxed coincidences.

For $k=3$, \eqref{eq:extreme-bond} leaves exactly the four vincular
patterns displayed in the theorem.  In Claesson's generalized-pattern
notation, Lemma~21 of \cite{Claesson2001} proves $(b\text{--}a\text{--}c)\Coin(b\text{--}ac)$,
 that is, $213\Coin2\underline{13}$.

Reverse and complement give the other three identities.  Each of the four
underlying classical patterns is coincident with its boxed version by
\cref{prop:classical-coincidence}, proving all four length-three cases.

Finally let $k=4$.  Equation \eqref{eq:extreme-bond} gives twelve candidates.
The following six are impossible; the indicated cell is pointless and shaded
in the boxed mesh, but is not shaded in the vincular mesh, contradicting
\eqref{eq:same-enclosed-diagonals}.
\begin{center}
\small
\begin{tabular}{c@{\quad}c@{\qquad}c@{\quad}c}
\toprule
vincular pattern&cell&vincular pattern&cell\\
\midrule
$\underline{14}32$&$(2,1)$&$\underline{41}23$&$(2,3)$\\
$2\underline{14}3$&$(1,3)$&$3\underline{41}2$&$(1,1)$\\
$23\underline{41}$&$(2,1)$&$32\underline{14}$&$(2,3)$\\
\bottomrule
\end{tabular}
\end{center}
The six candidates not eliminated in this way are $\underline{14}23$, $23\underline{14}$, $2\underline{41}3$, $3\underline{14}2$, $32\underline{41}$, and $\underline{41}32$.

The remaining four noncoincidences have the following explicit certificates.
In the last column we list every occurrence of the underlying classical
pattern in the separator; the entry following the semicolon is the omitted
fifth entry, which lies in its bounding rectangle and therefore blocks it from
being boxed.
\begin{center}
\small
\begin{tabular}{c@{\qquad}c@{\qquad}c@{\qquad}l}
\toprule
vincular pattern&separator&vincular occurrence&boxed candidates; blockers\\
\midrule
$\underline{14}23$&$15324$&$1534$&$1534;2\quad 1524;3$\\
$23\underline{14}$&$24315$&$2415$&$2415;3\quad 2315;4$\\
$32\underline{41}$&$42351$&$4251$&$4251;3\quad 4351;2$\\
$\underline{41}32$&$51342$&$5132$&$5132;4\quad 5142;3$\\
\bottomrule
\end{tabular}
\end{center}
Thus each separator contains the indicated vincular pattern and avoids the
corresponding boxed pattern.  The third candidate is the coincidence proved
in Lemma \ref{lem:tightening}, and the fourth is its reversal.  This completes the
classification.
\end{proof}

There is also a natural coincidence with a genuinely bivincular pattern.  It
does not merely repackage the mesh as a set of position bonds.

\begin{corollary}
\label{cor:bivincular}
We have  $\Boxp(2413)\Coin 2\underline{41}3\Coin(2413,\{2\},\{2\})$.
\end{corollary}

\begin{proof}
By Lemma \ref{lem:tightening}, it is enough to show that among occurrences $a,b,c,d$ of
$2\underline{41}3$ ($c<a<d<b$ and $b$ and $c$ are adjacent) there is one satisfying $d=a+1$.  Choose such an occurrence
$a,b,c,d$ minimizing $d-a$.  If $a<e<d$ for some entry $e$ of the permutation,
then, because $b$ and $c$ are adjacent, $e$ lies either to the left of $b$ or
to the right of $c$.  In the first case $e,b,c,d$ is a new occurrence with a
smaller gap; in the second, $a,b,c,e$ is.  Both contradict minimality.  Hence
$d=a+1$.  The reverse implication is immediate.
\end{proof}

The three coincident meshes in Corollary \ref{cor:bivincular} are shown in
Figure~\ref{fig:three-meshes}.  The last has both a complete shaded column and a
complete shaded row, so it is genuinely bivincular.

\begin{figure}[ht]
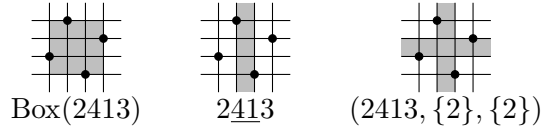

\centering
\begin{tabular}{c@{\qquad}c@{\qquad}c}
$\pattern{scale=0.68}{4}{1/2,2/4,3/1,4/3}{1/1,1/2,1/3,2/1,2/2,2/3,3/1,3/2,3/3}$
&$\pattern{scale=0.68}{4}{1/2,2/4,3/1,4/3}{2/0,2/1,2/2,2/3,2/4}$
&$\pattern{scale=0.68}{4}{1/2,2/4,3/1,4/3}{0/2,1/2,2/0,2/1,2/2,2/3,2/4,3/2,4/2}$\\
$\Boxp(2413)$&$2\underline{41}3$&$(2413,\{2\},\{2\})$
\end{tabular}
\caption{A boxed, a vincular and a genuinely bivincular presentation of the
same containment condition.}
\label{fig:three-meshes}
\end{figure}

\begin{theorem}
\label{thm:bivincular-bound}
Let $\tau\in\mathfrak S_k$, and let $B$ be a bivincular pattern of length
$k$.  If $\Boxp(\tau)\Coin B$, then $k\le4$.  In particular, boxed--bivincular coincidences do not occur in
arbitrarily large lengths.
\end{theorem}

\begin{proof}
Write $B=(\sigma,X,Y)$, where $X,Y\subseteq[k-1]$ are respectively its sets
of position and value bonds.  By Proposition \ref{prop:coincidence-invariants},
coincidence implies
\begin{equation}\label{eq:bivincular-enclosed}
 \sigma=\tau
 \qquad\text{and}\qquad
 \operatorname{enc}(\Boxp(\tau))=\operatorname{enc}(B).
\end{equation}
If $X=Y=\varnothing$, the result follows from
Proposition \ref{prop:classical-coincidence}.  If exactly one of $X,Y$ is nonempty, it
follows from \cref{thm:vincular-classification}, after applying inversion when
the bonds are value bonds.  We may therefore assume that $X$ and $Y$ are both
nonempty.

Fix $c\in X$.  The complete shaded column $c$ of $B$ contains the boundary
squares $(c,0)$ and $(c,k)$, neither of which is shaded in the boxed mesh.
Equality in \eqref{eq:bivincular-enclosed} implies that neither boundary square
can be pointless.  The first must touch a plot point of value $1$, and the
second a plot point of value $k$.  Hence $\{\tau_c,\tau_{c+1}\}=\{1,k\}$. The values $1$ and $k$ occur only once, so $X=\{c\}$.  Applying the same
argument after inversion gives, for the unique $d\in Y$,
\begin{equation}\label{eq:bivincular-extreme-row}
 \{\tau^{-1}(d),\tau^{-1}(d+1)\}=\{1,k\},
 \qquad Y=\{d\}.
\end{equation}
Thus the mesh of $B$ consists of one complete column $c$ and one complete row
$d$.

Every internal square outside row $d$ and column $c$ must touch a point of
$G(\tau)$.  Otherwise it would be a pointless enclosed diagonal of
$\Boxp(\tau)$ absent from $B$, contradicting~\eqref{eq:bivincular-enclosed}.

Suppose first that $k\ge6$ and $c\le k-2$.  Apart from the square in row $d$,
internal column $c+1$ contains $k-2\ge4$ squares that must touch
$G(\tau)$.  They can touch only the plot points in positions $c+1$ and
$c+2$.  By $\{\tau_c,\tau_{c+1}\}=\{1,k\}$, the first of these has value
$1$ or $k$ and touches only one internal square of the column; the second
touches at most two.  This is impossible, so $c=k-1$.  The analogous argument
in column $c-1$, when $c\ge2$, gives $c=1$, a contradiction.  Hence $k\le5$.

It remains to exclude $k=5$.  Complementing if necessary, assume $\tau_c=1$ and $\tau_{c+1}=5$.

If $c\le3$, the three squares of column $c+1$ outside row $d$ must all be
touched.  The point $(c+1,5)$ touches only row $4$, while the point in position
$c+2$, writing $q=\tau_{c+2}$, touches rows $q-1$ and $q$ when these are
internal.  For the touched rows to contain $[4]\setminus\{d\}$, either
$q=2$ and $d=3$, or $q=3$ and $d=1$; in particular, $d\in\{1,3\}$. Similarly, if $c\ge2$, column $c-1$ is touched by $(c,1)$ in row $1$ and by
the point in position $c-1$.  If $q=\tau_{c-1}$, the available rows are
$\{1,q-1,q\}\cap[4]$.  Covering the three required squares forces $d\in\{2,4\}$. Therefore $c$ cannot lie in $\{2,3\}$, and so $c\in\{1,4\}$.

After reversing and, if necessary, complementing, we may take $c=1$, $\tau_1=1$, and $\tau_2=5$. Equation \eqref{eq:bivincular-extreme-row} says that the values $d,d+1$ occupy
the first and last positions.  It follows that $d=1$ and $\tau_5=2$.
The preceding column-$2$ argument then forces $\tau_3=3$, and hence
$\tau=15342$.

The mesh of $B$ now contains the two boundary squares $D=\{(0,1),(1,0)\}$. They form a proper enclosed SE-diagonal surrounding the plot point $(1,1)$.
Neither square is shaded in $\Boxp(15342)$, so $D$ is absent from its set of
enclosed diagonals.  This contradicts \eqref{eq:bivincular-enclosed} and rules
out $k=5$.
\end{proof}

\section{The boxed mesh pattern
\texorpdfstring{$123$}{123}}\label{sec:123}

No closed formula is known for the number of permutations avoiding
$\B(123)$.  In this section we give an exact maximum-insertion identity,
use it to extend the known sequence, and prove a subfactorial upper bound.
We then obtain a pattern-independent first-moment formula; in particular, the
mean number of $\B(123)$ occurrences is $(n+2)H_n-3n\sim n\log n$, where
$H_n$ is the $n$-th harmonic number.

\subsection{Exact enumeration}

Let $\mathcal A_n=\Av_n(\B(123))$ and $a_n=|\mathcal A_n|$. For a word $w=w_1\cdots w_t$ with distinct entries, set $R_i(w)=|\{j:i<j\le t,\ w_j>w_i\}|$ and $\mathcal Q_t=\{w\in\mathfrak S_t:R_i(w)\ne1\text{ for all }i\}$.
 We call the gap after the first $t$ entries of $\pi$ {\em active} if insertion of the
new maximum in that gap preserves boxed $123$ avoidance.

\begin{theorem}\label{thm:active}
Let $\pi\in\mathcal A_n$, and let $w=\pi_1\cdots\pi_t$.  The gap following
$w$ is active if and only if $\st(w)\in\mathcal Q_t$.  Moreover, $|\mathcal Q_0|=1$ and $|\mathcal Q_t|=(t-1)!$ for $t\ge1$. Consequently,
\begin{equation}\label{eq:activeformula}
 a_{n+1}=\sum_{\pi\in\mathcal A_n}\sum_{t=0}^{n}
 \one\!\left[\st(\pi_1\cdots\pi_t)\in\mathcal Q_t\right].
\end{equation}
Here \(\one[\mathcal E]\) denotes the indicator of a statement
\(\mathcal E\): it equals \(1\) when \(\mathcal E\) is true and \(0\)
otherwise.  
\end{theorem}

\begin{proof}
Any new boxed $123$ occurrence must use the inserted maximum as its last and
largest selected entry.  If its first selected entry is $w_i$, the boxed
condition says that exactly one entry to the right of $w_i$ in the prefix $w$
is larger than $w_i$: this unique entry is the middle selected point.  This
condition is also sufficient, since every other entry between $w_i$ and the
new maximum is then below $w_i$ and hence outside the bounding rectangle.
Thus the gap is inactive precisely when some $R_i(w)$ equals one.

For the enumeration, recall that the \emph{Lehmer code} of
\(w=w_1\cdots w_t\in\mathfrak S_t\) is
\[
 \ell_i(w)=|\{j:i<j\le t,\ w_j<w_i\}|,
 \qquad 1\le i\le t.
\]
The map $w\longmapsto(\ell_1(w),\ldots,\ell_t(w))$ is a bijection between \(\mathfrak S_t\) and the set of vectors satisfying
$0\le \ell_i\le t-i$. Indeed, the permutation can be reconstructed successively by choosing \(w_i\)
to be the \((\ell_i+1)\)-st smallest of the values not yet used.

Since every entry to the right of \(w_i\) is either smaller or larger than
\(w_i\), we have $R_i(w)=t-i-\ell_i(w)$. Thus \((R_1(w),\ldots,R_t(w))\), which we call the
\emph{complementary Lehmer code}, is also a bijective code in which the
coordinates may independently take the values
\[
 R_i(w)\in\{0,1,\ldots,t-i\}.
\]
For \(i<t\), excluding the value \(1\) leaves \(t-i\) choices, while
\(R_t(w)=0\) has one choice.  Therefore
$
 |\mathcal Q_t|
 =\prod_{i=1}^{t-1}(t-i)
 =(t-1)!.
$
\end{proof}

Iterating \eqref{eq:activeformula} by maximum insertion provides an efficient
recursive generation and gives
\begin{equation}\label{eq:123data}
\begin{aligned}
 (a_n)_{n=0}^{17}={}&1,1,2,5,15,51,194,810,3675,17935,\\
                    &93481,517129,3021133,18559966,119468349,\\
                    &803213656,5624984106,40932781511.
\end{aligned}
\end{equation}
The terms through $n=12$ agree with OEIS entry A201168
\cite{OEIS201168}, which listed no further terms; the five terms $a_{13},\ldots,a_{17}$ are computed here.
In particular $a_6=194$, which differs from the value recorded in
\cite{AKV2013}; we have verified $a_6=194$ by exhaustive inspection of all
$720$ permutations of length six.  

Finding a closed coefficient formula for the sequence in \eqref{eq:123data} remains an open problem.  One possible route is to refine $\mathcal A_n$ by the set of prefix lengths
$t$ for which the complementary Lehmer code avoids the digit one.
\Cref{thm:active} identifies exactly the statistic that a successful refinement
must track.

\subsection{A subfactorial upper bound}

The factorial lower bound for $a_n$ rules out the Stanley--Wilf property, but
it does not determine how close the avoidance sequence is to $n!$.  The next
result gives a genuinely subfactorial upper bound.

\begin{theorem}\label{thm:box123-upper}
Let $0<x\le1$ satisfy
\begin{equation}\label{eq:box123-x-condition}
 2x+\sqrt{x}\le x(1+\sqrt{x})^2.
\end{equation}
Then, for every $n\ge1$,
\begin{equation}\label{eq:box123-parametric-bound}
 a_n\le 2^{5n}n^{\,n\log_2(1+\sqrt{x})}.
\end{equation}
The smallest admissible value is $x^*=\left(2\cos\frac{\pi}{7}-1\right)^2$, and hence
\begin{equation}\label{eq:box123-upper-bound}
 a_n\le 2^{5n}n^{\beta n},
 \qquad
 \beta=\log_2\!\left(2\cos\frac{\pi}{7}\right)\approx0.8496.
\end{equation}
\end{theorem}

Consequently,
\begin{equation}\label{eq:box123-growth-exponent}
 \frac12\le
 \liminf_{n\to\infty}\frac{\log a_n}{n\log n}
 \le\limsup_{n\to\infty}\frac{\log a_n}{n\log n}
 \le\beta<1.
\end{equation}
Here the lower bound follows from the doubled-letter injection recalled in the
introduction, together with the fact that $a_n$ is nondecreasing.  Indeed,
inserting a new maximum in front of an avoider creates no boxed $123$
occurrence, because the new entry has no entry to its left and can only play
the last selected point; hence
$a_n\ge a_{2\lfloor n/2\rfloor}\ge\lfloor n/2\rfloor!$ for every $n$.  We call
$\limsup_{n\to\infty}\log a_n/(n\log n)$ the \emph{growth exponent} of
$\Av(\B(123))$; by \eqref{eq:box123-growth-exponent} it lies between
$\tfrac12$ and $\beta$.  The upper bound also shows that a uniformly random
permutation avoids $\B(123)$ with probability at most $n^{-(1-\beta)n+O(n/\log n)}$.

To prove the theorem, for a permutation $\pi$ let
$\operatorname{asc}(\pi)$ be its number of ascents, and, for $0<x\le1$, set
\[
 Z_n(x)=\sum_{\pi\in\mathcal A_n}x^{\operatorname{asc}(\pi)},
 \qquad
 \tau(x)=\frac{2x+\sqrt{x}}{(1+\sqrt{x})^2}.
\]
For $0<x\le1$ one has $0<\tau(x)\le1$.

\begin{lemma}\label{lem:box123-halving}
For $n\ge2$, with $h_1=\lfloor n/2\rfloor$ and
$h_2=\lceil n/2\rceil$, one has
\begin{equation}\label{eq:box123-weighted-recurrence}
 Z_n(x)\le
 \frac{2(1+\sqrt{x})^{n-1}}{\tau(x)}
 Z_{h_1}(\tau(x))Z_{h_2}(\tau(x)).
\end{equation}
\end{lemma}

\begin{proof}
For $\pi\in\mathfrak S_n$, call the entries with values at most $h_1$
\emph{low} and the remaining entries \emph{high}.  Let
$L\in\mathfrak S_{h_1}$ and $H\in\mathfrak S_{h_2}$ be the patterns of the
low and high subsequences.  Number the gaps around the low entries from $0$ to
$h_1$, and number the slots between consecutive high entries from $1$ to
$h_2-1$.  Define
\begin{align*}
 U&=\{i:\text{low gap $i$ contains a high entry}\},\\
 D&=\{j:\text{high slot $j$ contains a low entry}\}.
\end{align*}
The nonempty low gaps contain the maximal blocks of high entries.  Thus
$|D|=|U|-1$, and $\pi$ is recovered from $(L,H,U,D)$ by cutting $H$ at the
slots in $D$ and placing the resulting blocks in the gaps in $U$.

Put $T=|U\setminus\{0\}|$.  This is the number of adjacent low--high pairs,
and $T\ge|D|$.  Let $I_L$ be the number of ascents $i$ of $L$ for which
$i\notin U$, and let $I_H$ be the number of ascents $j$ of $H$ for which
$j\notin D$.  For $\pi\in\mathcal A_n$ the following elementary restrictions
hold:
\begin{enumerate}[label=\textup{(\roman*)}]
\item $L\in\mathcal A_{h_1}$ and $H\in\mathcal A_{h_2}$;
\item neither $L$ nor $H$ has two consecutive ascents;
\item if $i$ is an ascent of $L$, then it is not the case that
      $i\notin U$ and $i+1\in U$;
\item if $j\ge2$ is an ascent of $H$, then it is not the case that
      $j-1\in D$ and $j\notin D$.
\end{enumerate}
For (i), a point inside the rectangle of an occurrence among the low entries
would itself be low, and similarly for the high entries.  For (ii), three
entries occupying consecutive positions of $L$ and increasing in value form a
boxed occurrence in $\pi$: between them there are no other low entries, while
all intervening high entries lie above their vertical range.  The argument for $H$ is
analogous.  In (iii) the entries $L_i,L_{i+1}$ followed by the first high
entry in gap $i+1$ occupy three consecutive positions of $\pi$ and increase;
in (iv), use the last low entry of slot $j-1$ followed by $H_j,H_{j+1}$.
These triples would be boxed occurrences, proving the restrictions.

Every adjacent ascent of $\pi$ is either low--high, low--low, or high--high,
and therefore
\begin{equation}\label{eq:box123-ascent-split}
 \operatorname{asc}(\pi)=T+I_L+I_H.
\end{equation}
Let $E_L$ be the set of ascents of $L$, and let $E_H$ be the set of ascents
$j\ge2$ of $H$.  From \eqref{eq:box123-ascent-split}, $T\ge|D|$, and
$0<x\le1$, we obtain $x^{\operatorname{asc}(\pi)}\le f(U)g(D)$, where
 $$f(U)=x^{|U\setminus\{0\}|/2}
       x^{|\{i\in E_L:i\notin U\}|},\ \ \ \ 
 g(D)=x^{|D|/2}
       x^{|\{j\in E_H:j\notin D\}|}.$$
If $1$ is an ascent of $H$, its contribution has been omitted from $g$; since
$x\le1$, this only enlarges the right-hand side.

By (ii), the pairs $\{i,i+1\}$ with $i\in E_L$ are disjoint.  In summing
$f(U)$ over all subsets $U\subseteq\{0,1,\ldots,h_1\}$ allowed by (iii), the
index $0$ contributes $2$, each
index in no such pair contributes $1+\sqrt{x}$, and each pair contributes $x+\sqrt{x}+x=2x+\sqrt{x}$; the fourth membership pattern is excluded by (iii).  Hence
\[
 \sum_U f(U)=2(1+\sqrt{x})^{h_1}\tau(x)^{|E_L|}.
\]
The same calculation, carried out over all subsets
$D\subseteq\{1,\ldots,h_2-1\}$ allowed by (iv), gives
\[
 \sum_D g(D)=(1+\sqrt{x})^{h_2-1}\tau(x)^{|E_H|}.
\]
The encoding by $(L,H,U,D)$ is injective, and every $\pi\in\mathcal A_n$
yields a pair $(U,D)$ allowed by (iii) and (iv) and satisfying
$|D|=|U|-1$.  Dropping that last relation enlarges the set of encodings summed
over, so summing over all allowed pairs $(U,D)$ gives
\begin{align*}
 Z_n(x)
 &\le 2(1+\sqrt{x})^{n-1}
 \sum_{L\in\mathcal A_{h_1}}\sum_{H\in\mathcal A_{h_2}}
 \tau(x)^{|E_L|+|E_H|}\\
 &\le \frac{2(1+\sqrt{x})^{n-1}}{\tau(x)}
 Z_{h_1}(\tau(x))Z_{h_2}(\tau(x)),
\end{align*}
because $|E_L|+|E_H|\ge
\operatorname{asc}(L)+\operatorname{asc}(H)-1$ and $\tau(x)\le1$.
\end{proof}

\begin{proof}[Proof of \cref{thm:box123-upper}]
Condition \eqref{eq:box123-x-condition} is precisely $\tau(x)\le x$.
Since $Z_h(y)$ is increasing in $y$, Lemma \ref{lem:box123-halving} yields
\[
 Z_n(x)\le\frac2{\tau(x)}(1+\sqrt{x})^{n-1}
 Z_{h_1}(x)Z_{h_2}(x).
\]
Put
\[
 F(n)=\log_2 Z_n(x),\qquad
 \gamma_x=\log_2(1+\sqrt{x}),\qquad
 c=\log_2\frac2{x^*}<1.64.
\]
The function $\tau$ is increasing on $(0,1]$, so
$\tau(x)\ge\tau(x^*)=x^*$ for admissible $x$.  Hence $F(1)=0$ and $F(n)\le \gamma_x n+c+F(h_1)+F(h_2)$. Induction gives
\[
 F(n)\le\gamma_x n\lceil\log_2n\rceil+c(2n-1).
\]
Indeed, both $h_i$ are at most
$2^{\lceil\log_2n\rceil-1}$, so substituting the two inductive bounds cancels
one halving level and gives the displayed inequality.

Since $\operatorname{asc}(\pi)\le n-1$ and $x\le1$, we have
$a_n\le x^{-n}Z_n(x)$.  Using
$\lceil\log_2n\rceil\le\log_2n+1$,
$\gamma_x\le1$, and $\log_2(1/x^*)<0.64$, we obtain
\[
 \log_2 a_n
 \le \gamma_x n\log_2n+
       (\gamma_x+3.28+0.64)n
 \le \gamma_x n\log_2n+5n,
\]
which proves \eqref{eq:box123-parametric-bound}.

Finally, with $s=\sqrt{x}$, condition
\eqref{eq:box123-x-condition} becomes
$s^3+2s^2-s-1\ge0$.  On $0<s\le1$ its unique threshold is
$s^*=2\cos(\pi/7)-1$: after setting $t=1+s$, the corresponding equation is
$t^3-t^2-2t+1=0$.  Its three roots are
$2\cos(\pi/7)$, $2\cos(3\pi/7)$, and $2\cos(5\pi/7)$, of which only the
first lies in $(1,2]$.  That $s^*$ is the only sign change on $(0,1]$ follows
from a direct inspection: writing $p(s)=s^3+2s^2-s-1$, the derivative
$p'(s)=3s^2+4s-1$ vanishes only at $s=(\sqrt7-2)/3=0.2153\ldots$, so $p$
decreases on $(0,(\sqrt7-2)/3)$ and increases on $((\sqrt7-2)/3,1]$; since
$p(0)=-1<0<1=p(1)$, the polynomial $p$ is negative on $(0,s^*)$ and
nonnegative on $[s^*,1]$.  Thus $x^*=(s^*)^2$ and
$1+\sqrt{x^*}=2\cos(\pi/7)$, proving \eqref{eq:box123-upper-bound}.
\end{proof}

\subsection{Average occurrence counts for boxed patterns}

For $\tau\in\mathfrak S_k$ define
\begin{equation}\label{eq:general-distribution-polynomial}
 F_{n,\tau}(q)=\sum_{\pi\in\mathfrak S_n}
 q^{N_{\B(\tau)}(\pi)},
\end{equation}
where $N_{\B(\tau)}(\pi)$ is the number of occurrences of $\B(\tau)$ in
$\pi$.  Its constant term is the avoidance number: $F_{n,\tau}(0)=|\Av_n(\B(\tau))|$. For uniformly random $\pi\in\mathfrak S_n$, its derivative satisfies
\begin{equation}\label{eq:first-moment-derivative}
 \frac{F'_{n,\tau}(1)}{n!}
 =\mathbb E\bigl[N_{\B(\tau)}(\pi)\bigr].
\end{equation}
The full occurrence distribution for $\B(123)$ is open, but the first moment
has a simple closed form.  In fact, it is independent of the choice of the
underlying pattern of a fixed length.

\begin{proposition}\label{prop:boxed-first-moment}
Let $k\ge2$, let $\tau\in\mathfrak S_k$, and write
$N_{\B(\tau)}(\pi)$ as in \eqref{eq:general-distribution-polynomial}.  If
$\pi$ is uniformly distributed over $\mathfrak S_n$, where $n\ge k$, then
\begin{equation}\label{eq:boxed-first-moment}
 \mathbb E\bigl[N_{\B(\tau)}(\pi)\bigr]
 =\frac{(n+k-1)(H_n-H_{k-2})-2(n-k+2)}{(k-2)!},
\end{equation}
where $H_m=\sum_{j=1}^m1/j$ is the \(m\)-th harmonic number and $H_0=0$.  In particular,
\begin{equation}\label{eq:box123-mean}
 \mathbb E\bigl[N_{\B(123)}(\pi)\bigr]
 =(n+2)H_n-3n.
\end{equation}
Moreover, for fixed $n$ and $k$, the expectation in
\eqref{eq:boxed-first-moment} is independent of the choice of
$\tau\in\mathfrak S_k$: all boxed patterns of length $k$ have the same first
moment on $\mathfrak S_n$.
\end{proposition}

\begin{proof}
Fix selected positions
$I=\{i_1<\cdots<i_k\}$ and put $d=i_k-i_1$.  The block from $i_1$ to
$i_k$ has length $d+1$, and its relative order is uniformly distributed over
$\mathfrak S_{d+1}$.  The selected entries form an occurrence of
$\B(\tau)$ exactly when they receive $k$ consecutive ranks in this block,
arranged in the order $\tau$.  There are $d-k+2$ choices for this interval of
ranks, after which the other $d-k+1$ ranks may be arranged arbitrarily.
Consequently,
\[
 \Pr\bigl(I\text{ is a }\B(\tau)\text{-occurrence}\bigr)
 =\frac{(d-k+2)!}{(d+1)!}.
\]
For a fixed span $d$, there are
$(n-d)\binom{d-1}{k-2}$ possible sets $I$.  Linearity of expectation therefore
gives
\begin{align*}
 \mathbb E\bigl[N_{\B(\tau)}(\pi)\bigr]
 &=\sum_{d=k-1}^{n-1}(n-d)\binom{d-1}{k-2}
       \frac{(d-k+2)!}{(d+1)!}\\
 &=\frac1{(k-2)!}\sum_{d=k-1}^{n-1}
       \frac{(n-d)(d-k+2)}{d(d+1)}\\
 &=\frac{(n+k-1)(H_n-H_{k-2})-2(n-k+2)}{(k-2)!}.
\end{align*}
Setting $k=3$ proves \eqref{eq:box123-mean}.
\end{proof}

\begin{remark}
Since
$H_n=\log n+\gamma+O(1/n),$
where $\gamma$ is the Euler--Mascheroni constant,
\eqref{eq:box123-mean} gives $\mathbb E\bigl[N_{\B(123)}(\pi)\bigr]
 =n\log n+(\gamma-3)n+O(\log n)$. Thus a uniformly random permutation of length $n$ has $n\log n+O(n)$ boxed
$123$ occurrences on average.  This identifies the natural
scale of the statistic, although a first-moment estimate alone does not imply
concentration or determine the probability of avoidance.
\end{remark}

\section{The remaining orbit
\texorpdfstring{$\{2143,3412\}$}{\{2143,3412\}}}\label{sec:remaining}

There is a useful one-sided analogue of the tightening lemma from
\cref{sec:length4}.  Recall that an occurrence of
$2\underline{14}3$ has the form $a,x,y,d$ in positional order, where $x$ and
$y$ are adjacent and $x<a<d<y$.

\begin{lemma}\label{lem:one-sided-tightening}
Every permutation containing $\B(2143)$ contains
$2\underline{14}3$.  Equivalently,
\[
 \Av_n(2\underline{14}3)\subseteq \Av_n(\B(2143))
 \qquad(n\ge0).
\]
\end{lemma}

\begin{proof}
Let $a,b,c,d$ be a boxed $2143$ occurrence in positional order, so $b<a<d<c$.
Every entry strictly between $b$ and $c$ in position lies either below $b$ or
above $c$; an entry with value strictly between them would lie in the empty
bounding rectangle.  Read the permutation from $b$ to $c$, regarding $b$ as
belonging to the lower region and $c$ as belonging to the upper region.  At
some point there is an adjacent transition $x,y$ from the lower region to the
upper region.  Hence $x\le b<a<d<c\le y$. The inequalities relevant to the four distinct entries $a,x,y,d$ are strict,
so these entries form an occurrence of $2\underline{14}3$.
The avoidance-class inclusion is the contrapositive.
\end{proof}

\begin{remark}
The reverse implication in Lemma \ref{lem:one-sided-tightening} is false.  For
example, $24153$ contains $2\underline{14}3$ in the entries $2,1,5,3$, but
avoids $\B(2143)$: the intervening entry $4$ obstructs the only candidate.
This is also forced by the classification in
\cref{thm:vincular-classification}.  Moreover,
$2\underline{14}3$-avoiders are the plane permutations and are enumerated by
the semi-Baxter numbers \cite{BouvelEtAl2018}.  Thus the lemma gives only the
exponential lower bound $\SB_n\le |\Av_n(\B(2143))|$. The inclusion has the wrong direction for an exponential upper bound and is
far too small to yield a factorial lower bound.  Its value is structural: it
identifies the semi-Baxter class as a canonical subclass of the remaining
boxed avoidance class and explains why tightening alone cannot resolve the
Stanley--Wilf question.
\end{remark}

We next observe that an injection from $\mathfrak S_m$ into avoiders whose
length is an affine function of $m$ is sufficient to disprove the
Stanley--Wilf property.  The observation applies to any mesh pattern,
not only to boxed ones.

\begin{proposition}
\label{prop:dilation-injection}
Let $P$ be a fixed mesh pattern.  Suppose that there is a
constant integer $c\ge1$, an integer $r$, and, for all sufficiently large
$m$, an injection
\[
 \Phi_m:\mathfrak S_m\longrightarrow\Av_{cm+r}(P),
\]
where $cm+r\ge0$.
Then $P$ does not have the Stanley--Wilf property.
\end{proposition}

\begin{proof}
Writing $A_n=|\Av_n(P)|$, injectivity gives $A_{cm+r}\ge m!$.  Hence, by
Stirling's formula,
\[
 A_{cm+r}^{1/(cm+r)}
 \ge (m!)^{1/(cm+r)}
 \sim \left(\frac me\right)^{1/c}\longrightarrow\infty.
\]
Thus $\limsup_{n\to\infty}A_n^{1/n}=\infty$.
\end{proof}

Consequently there is no need to insist on doubling: an embedding into
permutations of length $2m-1$, $3m$, $4m$, or, more generally, $cm+r$ for
fixed $c$ and $r$ would be equally decisive.  Increasing $c$ may provide more
room for blockers, although the extra entries can also create new occurrences.


Here is a concrete format in which such an injection may exist.  For
$\pi=\pi_1\cdots\pi_m\in\mathfrak S_m$ and
$\rho=\rho_1\cdots\rho_{m-1}\in\mathfrak S_{m-1}$, define the strict
interleaving
\begin{equation}\label{eq:interleave}
 J_\rho(\pi)=
 (2\pi_1-1,2\rho_1,2\pi_2-1,2\rho_2,\ldots,
  2\rho_{m-1},2\pi_m-1)\in\mathfrak S_{2m-1}.
\end{equation}
For $m=1$ the even word is empty.  Thus the prescribed odd entries occur in
their original order, and exactly one even entry is placed in each internal
gap.  Equivalently, one permutes the set
$\{2,4,\ldots,2m-2\}$ between the fixed odd entries
$2\pi_1-1,\ldots,2\pi_m-1$.

The odd-position entries recover the input without any auxiliary data:
\begin{equation}\label{eq:interleave-recovery}
 \pi_i=\frac{J_\rho(\pi)_{2i-1}+1}{2}.
\end{equation}
For example, if $\pi=231$ and $\rho=21$, then $J_{21}(231)=(3,4,5,2,1)$, which avoids $\B(2143)$, and its odd-position entries $3,5,1$ recover
$231$ by \eqref{eq:interleave-recovery}.

There is also a simple reason that formulations beginning with a forced
global minimum can be shortened.  If a permutation beginning with $1$ avoids
$\B(2143)$, then deleting this $1$ and reducing all other entries by one
preserves avoidance.  Conversely, adjoining a new global minimum on the left
preserves avoidance, since the minimum in the underlying pattern $2143$ is
its second entry and therefore cannot be played by the new leftmost entry.
This is a special case of the general observation that deleting a global
extremum cannot uncover a boxed occurrence: the deleted point lies outside
the vertical bounding interval of every occurrence formed by the remaining
entries.

\begin{conjecture}\label{conj:interleave}
Let $m\geq 1$. For every $\pi\in\mathfrak S_m$, there exists
$\rho\in\mathfrak S_{m-1}$ such that $J_\rho(\pi)$ avoids $\B(2143)$.
\end{conjecture}

The conjecture has been verified exhaustively for $m\le 8$. If true, choose
the lexicographically first admissible $\rho$; equation
\eqref{eq:interleave-recovery} then gives an explicit injection
\[
  \mathfrak S_m\hookrightarrow\Av_{2m-1}(\B(2143)),
  \qquad
  \pi\longmapsto J_{\rho(\pi)}(\pi),
\]
and hence, by Proposition~\ref{prop:dilation-injection}, the factorial lower
bound sought in \cite{AKV2013}. Conjecture~\ref{conj:interleave} is stronger than the growth conclusion that
it would imply, since factorial growth might arise from a different
construction. We therefore record the expected growth statement separately.

\begin{conjecture}\label{conj:2143-factorial}
The avoidance sequence of $\B(2143)$ has factorial growth, and hence
$\B(2143)$ does not have the Stanley--Wilf property.
\end{conjecture}

The following occurrence-preserving doubling provides a first step toward the
conjecture.  For $\pi=\pi_1\cdots\pi_m\in\mathfrak S_m$, define
\begin{equation}\label{eq:increasing-doubling}
 \Delta(\pi)=
 (2\pi_1-1)(2\pi_1)\,(2\pi_2-1)(2\pi_2)\cdots
 (2\pi_m-1)(2\pi_m).
\end{equation}
Thus each entry is replaced by an adjacent increasing pair of consecutive
values, and the entries in the odd positions recover $\pi$.

\begin{proposition}\label{prop:2143-doubling}
For every $\pi\in\mathfrak S_m$,
\begin{equation}\label{eq:2143-doubling-count}
 N_{\B(2143)}(\Delta(\pi))=N_{\B(2143)}(\pi).
\end{equation}
In particular, $\pi$ avoids $\B(2143)$ if and only if $\Delta(\pi)$ does.
\end{proposition}

\begin{proof}
Call the two entries replacing $x$ the $x$-block.  A $2143$ occurrence in
$\Delta(\pi)$ uses at most one entry from any block.  Indeed, two entries from
one block would be adjacent selected entries in increasing order.  The only
adjacent ascent in $2143$ is from its second selected entry, the minimum, to
its third, the maximum; two consecutive values cannot play these roles while
leaving two further selected values strictly between them.

Write $\varphi$ for the projection that sends a selected entry of the
$x$-block to $x$; it is well defined on the selected entries by the previous
paragraph, since each block contributes at most one of them.  It follows that
the four selected blocks of any $2143$ occurrence project under $\varphi$ to
four entries $a,b,c,d$ of $\pi$, in positional order, with
$b<a<d<c$.  If the occurrence in $\Delta(\pi)$ is boxed, then the projected
occurrence is boxed as well: any intervening entry $x$ of $\pi$ with
$b<x<c$ would place both entries of its block inside the bounding rectangle
in $\Delta(\pi)$.

Conversely, let $a,b,c,d$ be a boxed $2143$ occurrence of $\pi$.  It has the
unique boxed lift
\begin{equation}\label{eq:2143-unique-lift}
 2a,\quad 2b,\quad 2c-1,\quad 2d-1.
\end{equation}
To see uniqueness, choosing $2a-1$ would put its mate $2a$ immediately to the
right inside the rectangle; choosing $2b-1$ would put $2b$ inside; choosing
$2c$ would put $2c-1$ inside; and choosing $2d$ would put $2d-1$ inside.
The choices in \eqref{eq:2143-unique-lift} avoid these obstructions.  Their
unselected mates lie, respectively, to the left of the rectangle, below it,
above it, and to its right.  Any other point inside the lifted rectangle would
belong to an intervening block $x$ with $b<x<c$, contradicting that the
original occurrence was boxed.  Thus $\varphi$ and this unique lift are
mutually inverse maps between the boxed $2143$ occurrences of $\Delta(\pi)$
and those of $\pi$, proving \eqref{eq:2143-doubling-count}.
\end{proof}

Thus increasing doubling itself creates no additional boxed $2143$ occurrences.
It supplies a canonical alternating extension of length $2m$ whenever the
prescribed odd subsequence already avoids the pattern.  This is not yet the
strict length-$(2m-1)$ completion required by Conjecture \ref{conj:interleave}; for a
general input, the remaining task is to arrange the even entries so as to
destroy the projected occurrences without creating new ones.

\section{Boxed \texorpdfstring{$12$}{12} and strong Bruhat covers}
\label{sec:box12}

We finish the structural part of the paper with a distributional observation
about the smallest nontrivial boxed pattern, $\B(12)=
 \pattern{scale=0.52}{2}{1/1,2/2}{1/1}$. The avoidance sequence of this pattern is immediate: for each \(n\), the
decreasing permutation \(n(n-1)\cdots 1\) is its unique avoider.  However, to our
knowledge no general formula for its distribution is known, either in the
mesh-pattern literature or in the literature on the strong Bruhat order,
where only the extremal values and the mean of the statistic appear; see
Corollary \ref{cor:box12-moments}.  The statistic has a natural
interpretation in a different area of algebraic combinatorics.

For $\pi\in\mathfrak S_n$, let $\ell(\pi)=|\{(i,j):1\le i<j\le n,\ \pi_i>\pi_j\}|$ be its inversion number.  If $1\le i<j\le n$, let $\pi(i\leftrightarrow j)$
be obtained by interchanging the entries in positions $i$ and $j$.  The
\emph{strong Bruhat order} on $\mathfrak S_n$ is the transitive closure of the
covering relations
\[
 \pi\lessdot \pi(i\leftrightarrow j)
 \quad\Longleftrightarrow\quad
 \ell\bigl(\pi(i\leftrightarrow j)\bigr)=\ell(\pi)+1.
\]
The \emph{up-degree} of $\pi$ is $d_+(\pi)=|\{\sigma\in\mathfrak S_n:\pi\lessdot\sigma\}|$;
the down-degree $d_-(\pi)$ is defined analogously by counting
$\sigma\lessdot\pi$.  See Adin and Roichman \cite{AdinRoichman2006} for these
statistics and Bouvel, Ferrari and Tenner
\cite{BouvelFerrariTenner2024} for descriptions of Bruhat covering relations
in mesh-pattern language.

Let $b(\pi)$ denote the number of occurrences of $\B(12)$ in $\pi$.

\begin{proposition}\label{prop:box12-bruhat}
For every permutation $\pi$, $b(\pi)=d_+(\pi)$.
\end{proposition}

\begin{proof}
Choose $i<j$ with $a=\pi_i<\pi_j=b$.  Interchanging $a$ and $b$ changes the
inversion number by $1+2|\{k:i<k<j,\ a<\pi_k<b\}|$. It therefore gives an upward Bruhat cover exactly when no entry strictly
between positions $i$ and $j$ has value strictly between $a$ and $b$.  This is
precisely the empty-cell condition for the selected pair to be an occurrence
of $\B(12)$.
\end{proof}

\begin{remark}
Kitaev and Liese \cite[Theorem~12]{KitaevLiese2013} proved that $b$ has
Catalan distribution on $\Av_n(132)$: for $0\le k\le n-1$,
\[
 \bigl|\{\pi\in\Av_n(132):b(\pi)=k\}\bigr|
 =C(n-1,k)
 =\frac{n-k}{n}\binom{n-1+k}{n-1},
\]
where $C(r,k)$ denotes the $(r,k)$ entry of Catalan's triangle.  In view of
\cref{prop:box12-bruhat}, this also gives the distribution of the strong
Bruhat up-degree on the Catalan class $\Av_n(132)$.
\end{remark}

In the notation of \eqref{eq:general-distribution-polynomial}, define its
distribution polynomial by
\[
 D_n(q)=F_{n,12}(q)=\sum_{\pi\in\mathfrak S_n}q^{b(\pi)}.
\]
Complementation interchanges up-degree and down-degree.  Consequently the
results of Adin and Roichman
\cite[Propositions~2.1 and~2.9, and Theorem~4.1]{AdinRoichman2006}
translate as follows.

\begin{corollary}\label{cor:box12-moments}
For $n\ge1$, we have
\begin{align}
 \deg D_n(q)&=\left\lfloor\frac{n^2}{4}\right\rfloor,
 \label{eq:box12-degree}\\
 [q^{\lfloor n^2/4\rfloor}]D_n(q)
 &=
 \begin{cases}
  n,&n\text{ odd},\\
  n/2,&n\text{ even},
 \end{cases}
 \label{eq:box12-leading}\\
 \frac{D_n'(1)}{n!}
 &= (n+1)H_n-2n,
 \label{eq:box12-mean}
\end{align}
where \(H_n=\sum_{j=1}^n 1/j\) is the \(n\)-th harmonic number.
\end{corollary}

\begin{remark}
The familiar stronger algebraic properties do not hold in general.  Already
\[
 D_3(q)=1+2q+3q^2
\]
has nonreal zeros.  Moreover, for every $n\ge3$ the constant coefficient is
$1$, whereas \eqref{eq:box12-leading} shows that the leading coefficient is
greater than $1$.  Thus $D_n(q)$ is not palindromic, ruling out
$\gamma$-positivity in its usual palindromic sense.  Log-concavity also fails:
the polynomial $D_7(q)$ recorded in \cite{SuKitaevZhang2026} has $[q^{10}]D_7(q)=129$, $[q^{11}]D_7(q)=26$, $[q^{12}]D_7(q)=7$,
 and $26^2<129\cdot7$.  On the other hand, exhaustive enumeration gives a
unimodal coefficient sequence for every $D_n(q)$ with $n\le11$.
\end{remark}

This evidence suggests that unimodality may be the appropriate surviving
property.

\begin{conjecture}
\label{conj:box12-unimodal}
For every $n\ge0$, the coefficient sequence of $D_n(q)$ is unimodal.
\end{conjecture}

\begin{remark}
Interestingly, harmonic numbers also describe the complete distribution of a
rather different mesh pattern.  Let
\[
 p=
 \pattern{scale=0.50}{4}{1/2,2/1,3/3,4/4}%
 {0/0,0/1,0/2,0/3,0/4,1/0,1/4,2/0,2/4,3/0,3/4,%
  4/0,4/1,4/2,4/3,4/4}
\]
be the border mesh pattern studied in \cite{KitaevLiese2013}.  If $p_{n,k}$ denotes
the number of permutations in $\mathfrak S_n$ with exactly $k$ occurrences of
$p$, then \cite[Theorems~2 and~3]{KitaevLiese2013}, for $n\ge4$,
\[
 p_{n,k}=(n-2)!\bigl(H_{n-2}-H_k\bigr)
 \qquad(1\le k\le n-3),
\]
whereas $p_{n,0}=(n-2)!\bigl(H_{n-2}+n^2-2n+2\bigr)$.
\end{remark}

Equation \eqref{eq:box12-mean} is also the $k=2$ case of
Proposition \ref{prop:boxed-first-moment}.  The following direct derivation refines it by
the positional span of an occurrence.  Fix positions $i<j$ and put
$d=j-i$.  In the block $\pi_i\cdots\pi_j$ the two endpoint ranks are a
uniformly distributed ordered pair of distinct elements of $[d+1]$.  They
form a boxed $12$ occurrence exactly when the right endpoint has rank one
larger than the left endpoint.  There are $d$ such pairs out of $d(d+1)$, so
\[
 \Pr\bigl((i,j)\text{ is a }\B(12)\text{-occurrence}\bigr)=\frac1{d+1}.
\]
Linearity of expectation gives
\[
 \mathbb E[b(\pi)]
 =\sum_{d=1}^{n-1}\frac{n-d}{d+1}
 =(n+1)H_n-2n.
\]
Equivalently, the exponential generating function for the total number of
occurrences is
\begin{equation}\label{eq:box12-total-egf}
 \sum_{n\ge0}\frac{D_n'(1)}{n!}x^n
 =\frac{-\log(1-x)-x}{(1-x)^2}.
\end{equation}

Maximum insertion gives an exact identity for the full distribution, although
it does not close in the single sequence of polynomials $D_n(q)$.  For a word
$w=w_1\cdots w_t$ with distinct entries, let $\operatorname{rlmax}(w)
 =|\{i\in[t]:w_i>w_j\text{ for every }i<j\le t\}|$ be its number of right-to-left maxima, and set
$\operatorname{rlmax}(\varepsilon)=0$.

\begin{proposition}\label{prop:box12-insertion}
For $\pi\in\mathfrak S_n$ and $0\le t\le n$, let $\pi^{(t)}=\pi_1\cdots\pi_t\,(n+1)\,
             \pi_{t+1}\cdots\pi_n$. Then
\begin{equation}\label{eq:box12-insertion}
 b(\pi^{(t)})=b(\pi)+
 \operatorname{rlmax}(\pi_1\cdots\pi_t).
\end{equation}
Consequently,
\begin{equation}\label{eq:box12-distribution-insertion}
 D_{n+1}(q)=
 \sum_{\pi\in\mathfrak S_n}q^{b(\pi)}
 \sum_{t=0}^n q^{\operatorname{rlmax}(\pi_1\cdots\pi_t)}.
\end{equation}
\end{proposition}

\begin{proof}
Inserting the new maximum neither creates nor destroys an occurrence using
only old entries.  A new occurrence must end at $n+1$.  Its first entry
$\pi_i$, where $i\le t$, forms a boxed $12$ occurrence with $n+1$ exactly
when no later entry of the prefix $\pi_1\cdots\pi_t$ is larger than $\pi_i$,
that is, exactly when $\pi_i$ is a right-to-left maximum of the prefix.  This
proves \eqref{eq:box12-insertion}; summing over the $n+1$ insertion sites and
all parents proves \eqref{eq:box12-distribution-insertion}.
\end{proof}

Formula \eqref{eq:box12-distribution-insertion} shows both the promise and the
difficulty of maximum insertion: a closed recurrence must retain the joint
distribution of $b(\pi)$ with the right-to-left-maximum counts of all
prefixes.  We do not know a finite catalytic refinement that accomplishes
this, and the full polynomial $D_n(q)$ remains open.

\section{Concluding directions}

The growth classification is complete except for one length-four orbit.
Theorem~\ref{thm:global-growth}, together with the factorial construction of
\cite{AKV2013}, shows that no boxed pattern of length at least five has the
Stanley--Wilf property. In length four, twenty patterns have factorial lower
bounds, the orbit $\{2413,3142\}$ has semi-Baxter and hence exponential
growth, and only $\{2143,3412\}$ remains unresolved. We conjecture that this orbit has factorial
growth; see Conjecture~\ref{conj:2143-factorial}.

The coincidence results also reveal a sharp length boundary.
Theorem~\ref{thm:vincular-classification} shows that boxed--vincular
coincidences stop in length four, while Theorem~\ref{thm:bivincular-bound}
shows that the genuinely bivincular coincidence in
Corollary~\ref{cor:bivincular} is also maximal in length. The
interval-consecutive criterion provides a complementary structural viewpoint:
boxed $\tau$ avoidance is consecutive $\tau$ avoidance on every interval of
values. Moreover, Proposition~\ref{prop:interval-factor} transfers factorial
lower bounds from interval factors to larger boxed patterns.

For boxed $123$, Theorem~\ref{thm:active} gives an exact maximum-insertion
identity and five new terms, while Theorem~\ref{thm:box123-upper} bounds its
growth exponent by a constant below $0.85$. A closed enumeration remains
open. Proposition~\ref{prop:boxed-first-moment} also shows that all boxed
patterns of a fixed length have the same first moment; in particular, the
expected number of boxed $123$ occurrences is asymptotic to $n\log n$.

Finally, for boxed $12$, Proposition~\ref{prop:box12-bruhat} identifies the occurrence
statistic with up-degree in the strong Bruhat order. Although
real-rootedness, $\gamma$-positivity and log-concavity fail in small lengths,
the available data motivate Conjecture~\ref{conj:box12-unimodal}, which asserts
unimodality of the distribution polynomials.

\section*{Declaration of generative AI and AI-assisted technologies in the
manuscript preparation process}

During the preparation of this work the authors used
GPT (OpenAI) in order to assist with language editing, exploratory testing of
conjectures, proof development, and the drafting of parts of the paper.  After using this tool, the authors reviewed, corrected, and verified the
resulting text, arguments, and references, and take full responsibility for
the content of the manuscript.

\end{document}